\documentclass[11pt,a4paper,twoside]{article}

\usepackage{iftex}
\usepackage[T1]{fontenc}
\ifPDFTeX
  \usepackage[utf8]{inputenc}
\fi
\usepackage{lmodern}

\usepackage{amsmath,amsthm,amssymb,mathtools}

\usepackage{geometry}
\usepackage{xcolor}
\usepackage{microtype}
\usepackage{fancyhdr}

\usepackage[
  colorlinks=true,
  linkcolor=blue!45!black,
  citecolor=blue!45!black,
  urlcolor=blue!45!black
]{hyperref}

\allowdisplaybreaks[3]
\newtheoremstyle{paperplain}{0.65em}{0.65em}{\itshape}{}%
  {\bfseries}{.}{0.5em}{}
\theoremstyle{paperplain}
\newtheorem{theorem}{Theorem}[section]
\newtheorem{proposition}[theorem]{Proposition}
\newtheorem{lemma}[theorem]{Lemma}

\newtheoremstyle{paperremark}{0.65em}{0.65em}{\normalfont}{}%
  {\bfseries}{.}{0.5em}{}
\theoremstyle{paperremark}
\newtheorem{remark}[theorem]{Remark}

\newcommand{\runningauthors}{Guohuan Qiu and Jin Yan}
\newcommand{\runningtitle}{}
\let\originaltitle\title
\renewcommand{\title}[2][]{%
  \originaltitle{#2}%
  \if\relax\detokenize{#1}\relax
    \gdef\runningtitle{#2}%
  \else
    \gdef\runningtitle{#1}%
  \fi
}

\numberwithin{equation}{section}
\DeclareMathOperator{\tr}{tr}
\DeclareMathOperator{\dist}{dist}
\title[Hessian quotient equations]{\textbf{Interior estimates for Hessian quotient equations}}

\author{\textbf{Guohuan Qiu}\and\textbf{Jin Yan}}
\date{}

\makeatletter
\newcommand{\printfirstpagebottommatter}{%
  \begingroup
  \renewcommand{\thefootnote}{}%
  \renewcommand{\@makefntext}[1]{\noindent##1}%
  \footnotetext[0]{%
    \textit{Keywords.}
    Hessian quotient equation; interior Hessian estimate;
    nonlinear comparison; doubling method; Pogorelov estimate.
    \par\smallskip
    \textit{2020 Mathematics Subject Classification.}
    35J60, 35B45.%
  }%
  \endgroup
}
\makeatother

\newcommand{\printauthorinformation}{%
  \par\bigskip
  \begingroup
  \small
  \setlength{\parindent}{0pt}%

  \noindent
  \textbf{Guohuan Qiu}\par
  State Key Laboratory of Mathematical Sciences,
  Academy of Mathematics and Systems Science,
  Chinese Academy of Sciences,
  Beijing 100190, China; \\
  Institute of Mathematics,
  Academy of Mathematics and Systems Science,
  Chinese Academy of Sciences,
  55 Zhongguancun East Road, Beijing 100190, China.\par
  \textit{Email:}
  \href{mailto:qiugh@amss.ac.cn}{qiugh@amss.ac.cn}.

  \par\medskip

  \noindent
  \textbf{Jin Yan}\par
  Institute of Mathematics,
  Academy of Mathematics and Systems Science,
  Chinese Academy of Sciences,
  55 Zhongguancun East Road, Beijing 100190, China.\par
  \textit{Email:}
  \href{mailto:yanjin@amss.ac.cn}{yanjin@amss.ac.cn}.

  \endgroup
}

\hypersetup{%
  pdftitle={Interior estimates for Hessian quotient equations},
  pdfauthor={Guohuan Qiu and Jin Yan},
  pdfsubject={Classical interior estimates for Hessian quotient equations},
  pdfkeywords={Hessian quotient, interior estimate, nonlinear comparison}
}

\begin{document}

\maketitle
\printfirstpagebottommatter
\begin{abstract}
We prove interior Hessian estimates in arbitrary dimensions for admissible solutions of Hessian quotient equations with constant right-hand side, when the numerator and denominator differ in degree by one or two. Counterexamples show that this degree range is optimal. The proof combines a new nonlinear comparison function, a doubling maximum principle, and a Pogorelov estimate.
\end{abstract}

\section{Introduction}
Let $\sigma_j$ denote the unnormalized elementary symmetric
polynomial, with $\sigma_0=1$, and let
\[
 \Gamma_k=\{\lambda\in\mathbb R^n:
                 \sigma_j(\lambda)>0,\ 1\le j\le k\}.
\]
We use the same notation for the corresponding cone of real
symmetric matrices. We consider the constant Hessian quotient equations
\begin{equation}\label{eq:intro-equation}
 \frac{\sigma_k(D^2u)}{\sigma_l(D^2u)}=1,
 \qquad D^2u\in\Gamma_k,
 \qquad 0\le l<k\le n.
\end{equation}
The case $l=0$ is the $k$-Hessian equation. The operator
$(\sigma_k/\sigma_l)^{1/(k-l)}$ is elliptic and concave on $\Gamma_k$.

For the $k$-Hessian equations, interior Hessian estimates hold for
$k=1$ by linear theory and for $k=2$ in every dimension.
The two-dimensional case goes back to the pioneering work of
E.~Heinz \cite{Heinz} on the Monge--Amp\`ere equation.
In higher dimensions, the result was established in dimension three
by Warren--Yuan \cite{WarrenYuan2009} and Qiu \cite{Qiu},
in dimension four by Shankar--Yuan \cite{ShankarYuan4} and
Fan \cite{Fan}, and in arbitrary dimensions by Li--Wu
\cite{LiWuHessian}. For $k\ge3$, the counterexamples of
Pogorelov and Urbas \cite{Pogorelov,Urbas} show that such estimates
fail in general, even for convex solutions with constant right-hand side.

For Hessian quotients,  the constant equation $\sigma_2/\sigma_1=1$
follows in every dimension from the quadratic Hessian estimate
by a scalar Hessian translation. Interior Hessian estimates for admissible solutions of
$\sigma_3/\sigma_1=1$ in dimensions three and four follow from
the work of Wang and Yuan \cite{WangYuan2014}. Interior estimates have been obtained under convexity by
Lu \cite{LuThree,Lu}, Lu and Tsai \cite{LuTsai}, and
Li and Wu \cite{LiWuQuotient}, and under semiconvexity by
Mei and Yan \cite{MeiYan} and Dong and Zhang \cite{DZ}.  For degree gaps at least three,
Lu \cite{Lu} constructed singular convex solutions with smooth positive
right-hand sides. These results leave the following question for
general adjacent and gap-two quotients: does an interior Hessian
estimate hold in every dimension under natural admissibility alone?
We answer this question affirmatively.

\begin{samepage}
\begin{theorem}\label{thm:main}
Let $2\le k\le n$, $0\le l<k$, and $k-l\in\{1,2\}$.
Let $u\in C^4(B_1)$ satisfy
\[
 D^2u\in\Gamma_k,\qquad
 \frac{\sigma_k(D^2u)}{\sigma_l(D^2u)}=1
 \quad\text{in }B_1,
 \qquad
 \|u\|_{L^\infty(B_1)}+\|Du\|_{L^\infty(B_1)}\le K.
\]
Then
\[
 \sup_{B_{1/2}}|D^2u|\le C(n,k,K).
\]
\end{theorem}
\end{samepage}

\begin{remark}[Sharpness]\label{rem:sharpness}
The degree range in Theorem~\ref{thm:main} is sharp, even
for convex solutions with constant right-hand side.
For every $0\le l<k\le n$ with $k-l\ge3$, there is a sequence
of smooth convex solutions of $\sigma_k(D^2u)/\sigma_l(D^2u)=1$
on $B_1$, with uniformly bounded $C^1$ norms but unbounded
Hessians on $B_{1/2}$.
Such a sequence can be constructed by adapting the Dirichlet
approximation and barrier argument of Urbas
\cite[Section~3]{Urbas} to the Pogorelov-type quotient profile
of Lu \cite[Section~7]{Lu}.
\end{remark}

We briefly describe the main methods behind the interior theory.

\emph{Integral methods.}
The special Lagrangian estimates of Warren--Yuan,
Wang--Yuan, and Chen--Warren--Yuan
\cite{WarrenYuan2010,WangYuan2014,ChenWarrenYuan2009}
combine Jacobi inequalities, geometric mean value inequalities,
and integration by parts on the gradient graph.
Qiu developed this approach for scalar curvature equations
\cite{QiuScalar}, and Qiu--Zhou \cite{QiuZhou} extended
it to special Lagrangian curvature equations.

\emph{Legendre transforms and concavity inequalities.}
Yuan's Lewy rotation \cite{Yuan2002} and the Legendre--Lewy
method of Shankar--Yuan \cite{ShankarYuan2020} use convexity
or semiconvexity to change variables and obtain estimates
for the transformed equation.
Lu developed this approach for Hessian quotient equations,
first proving an interior estimate for $\sigma_3/\sigma_1$
in dimension three \cite{LuThree}, and then treating
$\sigma_n/\sigma_{n-1}$ and $\sigma_n/\sigma_{n-2}$
for convex solutions in arbitrary dimensions \cite{Lu}.
A key contribution of the latter work is a new concavity
inequality that controls the third-order terms and yields
a Jacobi inequality for the logarithm of the largest
Hessian eigenvalue.
The Legendre transform then leads to a mean value inequality,
while an induction argument controls the integrals arising
from repeated integration by parts. An important structural advance was made by Guan--Sroka
\cite{GuanSroka}, who established a special concavity
inequality for $\sigma_n/\sigma_l$, $1\le l<n$, on the
positive cone.
Lu--Tsai \cite{LuTsai} extended the Legendre and integral
approach to general Hessian quotients with degree gaps one
and two, reducing the convex interior estimate to a suitable
concavity inequality.
This inequality was subsequently established for adjacent
quotients by Tsai \cite{Tsai} and for both degree gaps
by Li--Wu \cite{LiWuQuotient}.
For semiconvex admissible solutions, Dong--Zhang \cite{DZ}
proved a quantitative concavity inequality using a lifting
argument and a decomposition into tangential and radial
directions.
This verifies the structural condition needed in the
Lu--Tsai framework under semiconvexity and, together with
a Lewy--Legendre transform and integral estimates, yields
the corresponding interior Hessian bounds.
Legendre duality also plays a central role in the Pogorelov
estimates of Lu--Tsai \cite{LuTsaiPogorelov}.

\emph{Doubling inequalities and compactness arguments.}
A doubling maximum principle from the first author's paper \cite{Qiu} propagates a Hessian bound from
a small ball to a larger one. Shankar--Yuan \cite{ShankarYuan4} combine doubling with
a compactness argument by contradiction: almost-everywhere second
differentiability of the limit and Savin's small perturbation theorem imply a uniform Hessian bound for the approximating sequence on one small ball.
Doubling propagates this bound to contradict Hessian blow-up. 
Fan adapted this approach to variable right-hand sides.
Fung \cite{Fung} obtained convex quotient estimates
by a pointwise doubling argument without Legendre transforms
or integral estimates.

\emph{Maximum principles and comparison functions.} Guan--Qiu \cite[Theorem~4 and Eq.~(20)]{GuanQiu}
established an interior maximum-principle estimate for the
$\sigma_2$-Hessian equation. Their proof introduced $x\cdot Du-u$ as an auxiliary function
in the maximum-principle approach to these interior estimates.
This function subsequently became a key ingredient in the
first author's doubling maximum principle \cite{Qiu} and its later developments. The Pogorelov estimate of Chou--Wang \cite{ChouWang}
uses a comparison function as a cutoff. Mooney \cite{Mooney}
proposed constructing such a comparison function from a linearized
Dirichlet problem, while Li--Wu used a nonlinear Dirichlet comparison
in their proof of the quadratic Hessian estimate. Qiu--Yan
\cite{QiuYan} combined a two-surface comparison with a doubling
argument for the graphical scalar-curvature equation.

Our proof for $l=k-1$ has three steps. Write $F=\sigma_k/\sigma_{k-1}$.

\medskip
\noindent\textit{1. A new nonlinear comparison function.}
The natural extension of the Li--Wu comparison,
$F(D^2v)=1/2$ with $v=u$ on the boundary, yields uniform separation
but not the decaying linearized error needed for our Pogorelov
estimate. Section~\ref{sec:counterexample} records a matrix obstruction
and a four-dimensional failure of the comparison-gap estimate when
only the diameter and $C^1$ data are controlled. We instead solve
\[
 \begin{cases}
 \sigma_k(D^2v)-\sigma_{k-1}(D^2v)=-\varepsilon
     &\text{in }B_{7/8},\\
 v=u &\text{on }\partial B_{7/8},
 \end{cases}
 \qquad D^2v-\tau I\in\Gamma_k,
\]
where $\tau=1/[2(n-k+1)]$ and $\varepsilon>0$ is fixed
and sufficiently small. We choose this equation because, at the original solution $u$,
the linearization of $\sigma_k-\sigma_{k-1}$ is
$\sigma_{k-1}(D^2u)$ times that of $F$.
This suggests a linearized comparison error of order
$\varepsilon/\sigma_{k-1}(D^2u)$.
Inspired by Guan--Zhang's treatment of mixed Hessian equations
\cite[Example~2.1 and Proposition~2.2]{GuanZhang},
we use a scalar Hessian translation to obtain a concave elliptic
formulation on a suitable admissible branch.
The Dirichlet theorem of Caffarelli--Nirenberg--Spruck
\cite[Theorem~$2'$]{CNS} then provides the comparison solution. The key estimate is
\[
 F^{ij}(D^2u)(v-u)_{ij}
 \ge-\frac{k\varepsilon}{\sigma_{k-1}(D^2u)}.
\]
A second reason for this choice is that the linearization is a
combination of divergence-free Newton tensors. This structure is
essential for the subsequent integration-by-parts argument, which
yields a quantitative lower bound for $v-u$ in a small ball.

\medskip
\noindent\textit{2. Uniform separation.}
Two integrations by parts using the Newton tensors, together with
an interior gradient bound, imply a fixed positive gap on a small ball.
A doubling maximum principle from \cite{Qiu} propagates this gap to $B_{1/2}$.

\medskip
\noindent\textit{3. A Pogorelov estimate.}
We use a positive level of $v-u$ as a cutoff. Ordinary concavity
leads to a directional Jacobi-type inequality for $\log\lambda_1$; its unfavorable
terms are absorbed by the gradient-square term in the test function.
The decaying comparison error and
$\sigma_{k-1}(D^2u)\ge c(n,k)\lambda_1$ close the estimate.
For $l=k-2$, adding the quadratic term $t^2/2$ in a new variable
reduces the equation to the adjacent case in dimension $n+1$.

Section~\ref{sec:algebra} collects the algebraic estimates.
Sections~\ref{sec:comparison}--\ref{sec:separation} construct the
comparison solution and prove its gradient bound and uniform separation.
Sections~\ref{sec:jacobi}--\ref{sec:completion} establish the directional Jacobi-type inequality and complete the adjacent interior estimate.
Section~\ref{sec:gap-two} treats the gap-two Hessian equation, and
Section~\ref{sec:counterexample} records the counterexamples.

\section{Preliminaries}\label{sec:algebra}

Throughout Sections~\ref{sec:algebra}--\ref{sec:completion}, we assume
${l}=k-1$. Write
\[
 F(A)=\frac{\sigma_k(A)}{\sigma_{k-1}(A)},\qquad
 G(A)=\sigma_k(A)-\sigma_{k-1}(A).
\]
On $\Gamma_k$, $F$ is positive and homogeneous of degree one.
It is elliptic and concave; see, e.g., \cite{CNS}.
Derivatives with respect to matrix entries are denoted by
$F^{ij}$ and $G^{ij}$. We use the standard Newton--Maclaurin
and G{\aa}rding inequalities throughout. Constants denoted by
$c,C>0$ may change from line to line, with their dependencies
indicated.

We first collect some basic estimates used in the comparison
construction and the final maximum principle. In particular,
we verify that a fixed scalar translation preserves
admissibility for the solution $u$. Set
\begin{equation}\label{eq:shift}
 \tau=\frac1{2(n-k+1)}.
\end{equation}
\begin{lemma}\label{lem:quotient-algebra}
For $A\in\Gamma_k$,
\begin{equation}\label{eq:trace-F}
 \frac{n-k+1}{k}\le\sum_iF^{ii}(A)\le n-k+1.
\end{equation}
If $F(A)\ge1/2$, then, in an eigenbasis of $A$, for each
$i=1,\ldots,n$,
\begin{equation}\label{eq:direction-F}
 F^{ii}(A)(1+|\lambda_i|)^2\ge c(n,k),\qquad
 \lambda_i\le0\ \Longrightarrow\ F^{ii}(A)\ge c(n,k).
\end{equation}
If $F(A)=1$, the following properties hold:
\begin{enumerate}
\item[(i)]
\begin{equation}\label{eq:sigma-trace}
 \sigma_{k-1}(A)\ge c(n,k)\tr A.
\end{equation}

\item[(ii)] 
\begin{equation}\label{eq:energy-lower}
 \sum_iF^{ii}(A)\lambda_i^2\ge\frac1{n-k+1}.
\end{equation}

\item[(iii)] 
\begin{equation}\label{eq:shift-admissibility}
 A-\tau I\in\Gamma_k.
\end{equation}
\end{enumerate}
\end{lemma}

\begin{proof}
Newton's inequality yields
\[
 \sum_iF^{ii}(A)
 =(n-k+1)-(n-k+2)
 \frac{\sigma_k(A)\sigma_{k-2}(A)}{\sigma_{k-1}(A)^2}
 \in\left[\frac{n-k+1}{k},\,n-k+1\right],
\]
which proves \eqref{eq:trace-F}.

Compute in an eigenbasis of $A$ and fix $i$.
Since $\lambda|i\in\Gamma_{k-1}$, both
$\sigma_{k-1}(\lambda|i)$ and $\sigma_{k-2}(\lambda|i)$
are positive. Newton's inequality in dimension $n-1$ yields
\begin{equation}\label{eq:deleted-newton}
 \begin{aligned}
 F^{ii}(A)
 =
 \frac{\sigma_{k-1}(\lambda|i)^2
 -\sigma_{k-2}(\lambda|i)\sigma_k(\lambda|i)}
{\sigma_{k-1}(A)^2}\ge
 \frac{n}{k(n-k+1)}
 \left(\frac{\sigma_{k-1}(\lambda|i)}
{\sigma_{k-1}(A)}\right)^2.
 \end{aligned}
\end{equation}
Moreover,
\[
 \sigma_{k-1}(\lambda|i)
 -F(A)\sigma_{k-2}(\lambda|i)
 =\sigma_{k-1}(A)F^{ii}(A)>0.
\]
Thus, if $F(A)\ge1/2$ and $\lambda_i\ge0$,
\[
 \frac{\sigma_{k-1}(\lambda|i)}{\sigma_{k-1}(A)}
 =
 \left(
  1+\lambda_i
  \frac{\sigma_{k-2}(\lambda|i)}
       {\sigma_{k-1}(\lambda|i)}
 \right)^{-1}
 \ge\frac1{1+2\lambda_i}.
\]
If $\lambda_i\le0$, the same ratio is at least one.
Together with \eqref{eq:deleted-newton}, this proves
\eqref{eq:direction-F}.

We also prove an inequality which will be used later. Newton-Maclaurin's inequality implies
\[
 \sigma_{k-1}(\lambda|i)^{k-2}
 \le C(n,k)\sigma_{k-2}(\lambda|i)^{k-1}.
\]
Consequently,
\begin{equation}\label{eq:deleted-lower}
 \begin{aligned}
 \frac{\sigma_{k-1}(\lambda|i)^2}
      {\sigma_{k-2}(\lambda|i)}
 \ge c(n,k)
 \left(
  \frac{\sigma_{k-1}(\lambda|i)}
       {\sigma_{k-2}(\lambda|i)}
 \right)^k\ge c(n,k)F(A)^k.
 \end{aligned}
\end{equation}

We now assume $F(A)=1$ and prove (i)--(iii).
Newton-Maclaurin's inequality again yields
\begin{equation}\label{eq:ratio-chain}
 \frac{\sigma_j(A)}{\sigma_{j-1}(A)}
 \ge\frac{k(n-j+1)}{j(n-k+1)}F(A),
 \qquad 1\le j<k.
\end{equation}
Multiplying over $2\le j\le k-1$ proves
\eqref{eq:sigma-trace}.

By the Cauchy--Schwarz inequality and \eqref{eq:trace-F}, we have
\[
 \sum_iF^{ii}(A)\lambda_i^2
 \ge
 \frac{\bigl(\sum_iF^{ii}(A)\lambda_i\bigr)^2}
      {\sum_iF^{ii}(A)}
 =\frac1{\sum_iF^{ii}(A)}
 \ge\frac1{n-k+1},
\]
which proves \eqref{eq:energy-lower}.

Finally, along the admissible part of the path $A-tI$,
\eqref{eq:trace-F} implies
\[
 \frac{d}{dt}F(A-tI)
 =-\sum_iF^{ii}(A-tI)\ge-(n-k+1).
\]
Thus
\[
 F(A-tI)\ge1-(n-k+1)t\ge\frac12
 \qquad (0\le t\le\tau).
\]
The Newton--Maclaurin inequality shows
$F\le C(n,k)\sigma_k^{1/k}$ on $\Gamma_k$. Thus $F$ tends to zero
at every finite boundary point of $\Gamma_k$. The path cannot leave
the cone while $F\ge1/2$, proving \eqref{eq:shift-admissibility}.
\end{proof}

\section{A concave Dirichlet comparison}\label{sec:comparison}

We construct the comparison solution and prove the two linearized
inequalities used below. Only the shifted operator associated with
$\sigma_k-\sigma_{k-1}=-\varepsilon$ is needed.

Fix $\tau$ as in \eqref{eq:shift}, and for $0\le j\le k-1$, set
\begin{equation}\label{eq:alpha}
 \alpha_j=
 \binom{n-j}{k-1-j}\tau^{k-1-j}
 \left(1-\frac{(n-k+1)\tau}{k-j}\right)>0.
\end{equation}
In particular, $\alpha_{k-1}=1/2$. The identity
\[
 \sigma_{l}(B+\tau I)
 =\sum_{j=0}^{l}\binom{n-j}{{l}-j}
                \tau^{{l}-j}\sigma_j(B)
\]
shows
\begin{equation}\label{eq:shift-polynomial}
 G(B+\tau I)=\sigma_k(B)-\sum_{j=0}^{k-1}\alpha_j\sigma_j(B).
\end{equation}

For $0<\varepsilon<\alpha_0/2$, define
$\mathcal F_\varepsilon$ on $\Gamma_k$ by
\begin{equation}\label{eq:concave-operator}
 \mathcal F_\varepsilon(B)
 =\frac{\sigma_k(B)}{\sigma_{k-1}(B)}
  -\sum_{j=1}^{k-2}\alpha_j
       \frac{\sigma_j(B)}{\sigma_{k-1}(B)}
  -\frac{\alpha_0-\varepsilon}{\sigma_{k-1}(B)}.
\end{equation}
Equivalently,
\begin{equation}\label{eq:operator-identity}
 \mathcal F_\varepsilon(B)
 =\frac12+\frac{G(B+\tau I)+\varepsilon}{\sigma_{k-1}(B)}.
\end{equation}
For $0\le j\le k-2$,
$(\sigma_{k-1}/\sigma_j)^{1/(k-1-j)}$ is positive, concave,
and elliptic. For $m=k-1-j$, the scalar function $s\mapsto-s^{-m}$
is increasing and concave on $(0,\infty)$. Composition with
the corresponding degree-one quotient is therefore concave
and increasing in positive matrix directions. Thus $\mathcal F_\varepsilon$
is concave and elliptic on $\Gamma_k$.
This operator belongs to the class considered in
\cite[Proposition~2.2]{GuanZhang}; the argument above gives the
properties needed here directly.

\begin{proposition}\label{prop:comparison}
There is a fixed $\varepsilon=\varepsilon(n,k)>0$ and a smooth
solution $v$ on $\overline B_{7/8}$ of
\begin{equation}\label{eq:comparison-equation}
 \begin{cases}
 \sigma_k(D^2v)-\sigma_{k-1}(D^2v)=-\varepsilon
       &\text{in } B_{7/8},\\
 v=u&\text{on }\partial B_{7/8},
 \end{cases}
 \qquad D^2v-\tau I\in\Gamma_k.
\end{equation}
It satisfies $u<v\le\max_{\partial B_{7/8}}u$.
Write
\[
 L_u f=F^{ij}(D^2u)f_{ij},\qquad
 a^{ij}=\frac{G^{ij}(D^2v)}{\sigma_{k-1}(D^2v)},\qquad
 L_v f=a^{ij}f_{ij}.
\]
Then
\begin{align}
 L_u(v-u)&\ge-\frac{k\varepsilon}{\sigma_{k-1}(D^2u)},
 \label{eq:strong-u-comparison}\\
 L_v(v-u)&<0.
 \label{eq:strong-v-comparison}
\end{align}
Moreover, in a principal basis for $D^2v$,
\begin{gather}
 c(n,k)\le\sum_i a^{ii}\le C(n,k),\qquad
 a^{ii}(1+|v_{ii}|)^2\ge c(n,k),
 \label{eq:a-bounds}\\
 v_{ii}\le0\ \Longrightarrow\ a^{ii}\ge c(n,k),\qquad
 \frac12\le L_vv\le1,\qquad L_vv_{l}=0.
 \label{eq:a-identities}
\end{gather}
\end{proposition}

\begin{proof}
The operator $\mathcal F_\varepsilon$ is smooth, symmetric,
elliptic, and concave on $\Gamma_k$. Since $\mathcal F_\varepsilon\le F$,
at each finite boundary point we have
\[
 \limsup_{B\to\partial\Gamma_k}\mathcal F_\varepsilon(B)
 \le0<\frac12.
\]
Moreover, homogeneity of its terms leads to
\[
 \mathcal F_\varepsilon(tB)=tF(B)+O(t^{-1})\longrightarrow+\infty
 \qquad\text{as }t\to\infty,
\]
uniformly on compact subsets of $\Gamma_k$.
The ball satisfies the boundary cone condition. Thus
\cite[Theorem~$2'$]{CNS}, applied to $v-\tau|x|^2/2$ with boundary
value $u-\tau|x|^2/2$, yields a smooth admissible solution of
\eqref{eq:comparison-equation}.
Here $u$ is smooth near $\overline B_{7/8}$ by local ellipticity
and bootstrapping. The existence theorem may use higher norms of $u$;
its quantitative bounds are not used below.

By \eqref{eq:shift-admissibility}, $u-\tau|x|^2/2$ is admissible, and
\[
 \mathcal F_\varepsilon(D^2u-\tau I)
 =\frac12+\frac{\varepsilon}{\sigma_{k-1}(D^2u-\tau I)}
 >\frac12.
\]
The comparison principle shows $v>u$. Since $D^2v\in\Gamma_k$,
subharmonicity implies $v\le\max_{\partial B_{7/8}}u$, hence $|v|\le K$.

We next derive the two linearized comparison inequalities.
For this calculation, write $A=D^2u$ and $B=D^2v$.
Since $G(A)=0$, we have
$G^{ij}(A)=\sigma_{k-1}(A)F^{ij}(A)$.
Concavity of $\mathcal F_\varepsilon$ at $A-\tau I$, together with
\eqref{eq:operator-identity}, leads to
\begin{equation}\label{eq:comparison-key}
 \begin{aligned}
 \sigma_{k-1}(A)L_u(v-u)
 &\ge-\varepsilon+
 \varepsilon\frac{\sigma_{k-1}^{ij}(A-\tau I)(B-A)_{ij}}
                    {\sigma_{k-1}(A-\tau I)}
 \ge-k\varepsilon.
 \end{aligned}
\end{equation}
The last inequality follows from $A-\tau I,B-\tau I\in\Gamma_k$,
G{\aa}rding's inequality, and homogeneity:
\[
 \sigma_{k-1}^{ij}(A-\tau I)(B-\tau I)_{ij}>0,\qquad
 \sigma_{k-1}^{ij}(A-\tau I)(A-\tau I)_{ij}
 =(k-1)\sigma_{k-1}(A-\tau I).
\]
This proves \eqref{eq:strong-u-comparison}.

At $B-\tau I$ the numerator
$G(B)+\varepsilon$ vanishes. Concavity at this point yields
\[
 \frac{\varepsilon}{\sigma_{k-1}(A-\tau I)}
 \le\frac{G^{ij}(B)(A-B)_{ij}}
          {\sigma_{k-1}(B-\tau I)}.
\]
Consequently,
\[
 \begin{aligned}
 L_v(v-u)
 =\frac{G^{ij}(B)(B-A)_{ij}}{\sigma_{k-1}(B)}\le-
 \frac{\varepsilon\,\sigma_{k-1}(B-\tau I)}
      {\sigma_{k-1}(B)\sigma_{k-1}(A-\tau I)}
 <0,
 \end{aligned}
\]
which proves \eqref{eq:strong-v-comparison}.

We finally fix $\varepsilon$ and establish the coefficient
bounds needed for the gradient and separation estimates. Since $D^2v-\tau I\in\Gamma_k$, we have $\sigma_{k-1}(D^2v)\ge\binom n{k-1}\tau^{k-1}$. Choose 
\[
 0<\varepsilon<
 \min\left\{
 \frac{\alpha_0}{2},\,
 \frac{1}{2k}\binom{n}{k-1}\tau^{k-1}
 \right\}.
\]
Then
\begin{equation}\label{eq:f-v-range}
 \frac12\le F(D^2v)
 =1-\frac{\varepsilon}{\sigma_{k-1}(D^2v)}<1.
\end{equation}
In the notation of Lemma~\ref{lem:quotient-algebra}, 
\[
 a^{ii}
 =F^{ii}(D^2v)
 -\frac{\varepsilon\,\sigma_{k-2}(\lambda|i)}
       {\sigma_{k-1}(D^2v)^2}.
\]
By \eqref{eq:deleted-newton} and \eqref{eq:deleted-lower},
\[
 \begin{aligned}
 0\le
 \frac{F^{ii}(D^2v)-a^{ii}}{F^{ii}(D^2v)}
 &=
 \frac{\varepsilon\,\sigma_{k-2}(\lambda|i)}
      {\sigma_{k-1}(\lambda|i)^2
       -\sigma_{k-2}(\lambda|i)\sigma_k(\lambda|i)}\\
 &\le C(n,k)\varepsilon
 \frac{\sigma_{k-2}(\lambda|i)}
      {\sigma_{k-1}(\lambda|i)^2}\\
 &\le C(n,k)\varepsilon F(D^2v)^{-k}
 \le C(n,k)\varepsilon,
 \end{aligned}
\]
where the last inequality uses $F(D^2v)\ge1/2$.
Decreasing $\varepsilon(n,k)$ if necessary, we obtain
\begin{equation}\label{eq:a-F-comparison}
 \frac12F^{ii}(D^2v)\le a^{ii}\le F^{ii}(D^2v).
\end{equation}
Together with \eqref{eq:trace-F} and \eqref{eq:direction-F},
this proves \eqref{eq:a-bounds}.

Finally, homogeneity and the choice of $\varepsilon$ yield
\[
 \begin{aligned}
 L_vv
 &=
 \frac{k\sigma_k(D^2v)-(k-1)\sigma_{k-1}(D^2v)}
      {\sigma_{k-1}(D^2v)}\\
 &=1-\frac{k\varepsilon}{\sigma_{k-1}(D^2v)}
 \ge\frac12.
 \end{aligned}
\]
Differentiating $G(D^2v)=-\varepsilon$ yields
$G^{ij}(D^2v)v_{ij{l}}=0$, and hence $L_vv_{l}=0$.
\end{proof}

\section{Interior gradient and boundary localization}\label{sec:gradient}

Interior gradient estimates for Hessian quotient equations were studied
in \cite{ChenGradient}. The following proof for our comparison equation
uses only the coefficient bounds
in Proposition~\ref{prop:comparison} and the height bound for $v$.
We include its proof because the boundary gradient estimates in the Dirichlet existence theory do not give the required dependence on the lower-order data. We also record a boundary modulus for $v-u$.

\begin{lemma}\label{lem:gradient}
Let $v$ be the solution in Proposition~\ref{prop:comparison}. Then for every \(B_R(x_0)\Subset B_{7/8}\),
\[
 \sup_{B_{R/2}(x_0)}|Dv|
 \le C(n,k,K,R).
\]
There is also a bound
\begin{equation}\label{eq:boundary-modulus}
 0\le v(x)-u(x)
 \le C(n,k)K\,\dist(x,\partial B_{7/8})^{1/2}.
\end{equation}
\end{lemma}

\begin{proof}
For the interior gradient estimate, set
\[
 \rho=R^2-|x-x_0|^2,\qquad
 \phi(t)=-\frac18\log(M-t),\qquad
 M=1+\sup_{B_{7/8}}v.
\]
We have \(|v|\le K\),
\(\phi''=8(\phi')^2\), and $\frac1{8(1+2K)}\le\phi'\le\frac18$.

The function $\rho|Dv|e^{\phi(v)}$ vanishes on $\partial B_R(x_0)$.
If its maximum is zero, the estimate is immediate. Otherwise, it
has a positive interior maximum at $\bar x$, where we may use
\[
 \Psi=\log\rho+\log|Dv|+\phi(v).
\] Choose coordinates at $\bar x$ such that $D^2v(\bar x)=\operatorname{diag}(\lambda_1,\cdots,\lambda_n)$. All computations below are carried out at $\bar x$.

The critical equation is
\[
 0=\frac{\rho_i}{\rho}
   +\frac{v_{l} v_{{l} i}}{|Dv|^2}
   +\phi'v_i.
\]
It follows that
\begin{equation}\label{eq:gradient-first}
 \bigl(\lambda_i+\phi'|Dv|^2\bigr)v_i
 =-\frac{|Dv|^2}{\rho}\rho_i.
\end{equation}

Choose $i$ with $|v_i|\ge |Dv|/\sqrt n$.
If $\rho|Dv|\le32\sqrt nR(1+2K)$, the required bound at the
maximum point already holds. Otherwise,
\[
 \left|\frac{\rho_i}{\rho v_i}\right|
 \le\frac{2\sqrt nR}{\rho|Dv|}\le\frac{\phi'}2,
\]
and \eqref{eq:gradient-first} implies
\[
 \lambda_i=-|Dv|^2\left(\phi'+\frac{\rho_i}{\rho v_i}\right)
 \le-\frac{\phi'}2|Dv|^2<0.
\]
By \eqref{eq:a-identities}, $a^{ii}\ge c(n,k)$. Therefore
\begin{equation}\label{eq:gradient-coercivity}
 a^{pq}v_pv_q\ge a^{ii}v_i^2\ge c(n,k)|Dv|^2.
\end{equation}

Since $L_vv_{l}=0$, direct differentiation shows
\[
 L_v\log|Dv|
 =
 \frac{a^{ij}v_{{l} i}v_{{l} j}}{|Dv|^2}
 -2a^{ij}(\log|Dv|)_i(\log|Dv|)_j.
\]
Using the Cauchy--Schwarz inequality and the critical equation, we obtain
\[
 \begin{aligned}
 -2a^{ij}(\log|Dv|)_i(\log|Dv|)_j
 &\ge
 -4(\phi')^2a^{ij}v_iv_j
 -\frac4{\rho^2}a^{ij}\rho_i\rho_j.
 \end{aligned}
\]
Moreover,
\[
 L_v\phi(v)
 =\phi'L_vv+\phi''a^{ij}v_iv_j
\]
and
\[
 L_v\log\rho
 =\frac{L_v\rho}{\rho}
  -\frac1{\rho^2}a^{ij}\rho_i\rho_j.
\]
Combining the preceding estimates and \eqref{eq:a-bounds}, we find
\[
 \begin{aligned}
 0\ge L_v\Psi
 \ge
 \frac{a^{ij}v_{{l} i}v_{{l} j}}{|Dv|^2}
 +\bigl(\phi''-4(\phi')^2\bigr)a^{ij}v_iv_j+\phi'L_vv-\frac{C(n,k,R)}{\rho^2}.
 \end{aligned}
\]
The first term is nonnegative, while
$\phi'L_vv\ge0$. Equation \eqref{eq:gradient-coercivity} therefore implies
\[
 0\ge
 c(n,k)(\phi')^2|Dv|^2
 -\frac{C(n,k,R)}{\rho^2}.
\]
Thus we conclude that
\[
 \rho(\bar x)|Dv(\bar x)|
 \le C(n,k,K,R).
\]

Since $1\le M-v\le1+2K$, the oscillation of $\phi(v)$ is bounded
by a constant depending on $K$. The maximum property therefore leads to $\rho|Dv|\le C(n,k,K,R)$ in $B_R(x_0)$. On $B_{R/2}(x_0)$,
$\rho\ge3R^2/4$, proving the gradient estimate.

To estimate the decay near the boundary, we adapt the barrier
argument in \cite[Lemma~6.1]{QiuYan}. Set $R_0=\frac78$. For each $y\in\partial B_{R_0}$, define
\[
 w_y(x)
 =u(y)+K\sqrt{2(R_0^2-x\cdot y)},
 \qquad x\in\overline B_{R_0}.
\]
On $\partial B_{R_0}$ we have
\[
 w_y(z)=u(y)+K|z-y|\ge u(z)=v(z).
\]
The function $w_y$ is concave, since it is the square root of a
positive affine function plus a constant. Thus $\Delta w_y\le0$.
Since $v$ is subharmonic, the maximum principle shows $v\le w_y$
in $B_{R_0}$.

Now fix $x\in B_{R_0}$ and set
\[
 d=\dist(x,\partial B_{R_0}).
\]
Choose a nearest boundary point $y\in\partial B_{R_0}$.
Then
\[
 |x-y|=d,\qquad
 R_0^2-x\cdot y=R_0d.
\]
Using the Lipschitz bound for $u$, we obtain
\[
 \begin{aligned}
 0\le v(x)-u(x)
 \le w_y(x)-u(x)\le K\sqrt{2R_0d}+K|x-y|\le C K\,d^{1/2},
 \end{aligned}
\]
This proves \eqref{eq:boundary-modulus}.
\end{proof}

\section{Uniform separation}\label{sec:separation}

We first obtain a quantitative lower bound for the gap on a
small interior ball. The argument uses the divergence-free
Newton tensors and two integrations by parts. We then
propagate this bound by a single maximum-principle argument.

\begin{lemma}\label{lem:seed}
There are $\delta>0$, $0<r\le1/8$, depending only on $n,k,K$,
and $y_0\in\overline B_{1/8}$ such that
\[
 v-u\ge\delta\quad\hbox{in }B_r(y_0).
\]
\end{lemma}

\begin{proof}
Write $w=v-u$ and $f_t=(1-t)u+tv$.
For a fixed nonzero nonnegative cutoff $\zeta\in C_c^\infty(B_{1/8})$,
the divergence-free Newton tensors and two integrations by parts imply
\begin{equation}\label{eq:seed-ibp}
 \begin{aligned}
 \varepsilon\int\zeta
 &=-\int_0^1\int w\bigl(\sigma_k^{ij}(D^2f_t)
-\sigma_{k-1}^{ij}(D^2f_t)\bigr)\zeta_{ij}\,dt\\
 &\le C(n)\sup_{B_{1/8}}w
 \int_0^1\int_{B_{1/8}}\bigl(\sigma_{k-1}(D^2f_t)
+\sigma_{k-2}(D^2f_t)\bigr)\,dt.
 \end{aligned}
\end{equation}
Here $D^2f_t\in\Gamma_k$ by convexity of the cone, so both Newton
tensors are positive definite.

To bound the last integral, we show that if $D^2f\in\Gamma_k$ on $B_{1/2}$ and
$|Df|\le C$, then
\begin{equation}\label{eq:measure-bound}
 \int_{B_{1/4}}\sigma_j(D^2f)\le C(n,k,C),\qquad 1\le j\le k-1.
\end{equation}
For a nonnegative cutoff $\chi$, it follows from the divergence-free property that
\[
 j\int\chi\,\sigma_j(D^2f)
 =-\int\chi_a\sigma_j^{ab}(D^2f)f_b.
\]
The matrix $(\sigma_j^{ab})$ is positive definite and has trace
$(n-j+1)\sigma_{j-1}$. Hence
\[
 j\int\chi\,\sigma_j(D^2f)
 \le(n-j+1)\|D\chi\|_\infty\|Df\|_\infty
       \int_{\operatorname{supp}D\chi}\sigma_{j-1}(D^2f).
\]
Choose nested cutoffs between $B_{1/4}$ and $B_{1/2}$ and iterate
until $j=0$. This proves \eqref{eq:measure-bound}.

Apply this bound to $f_t$, whose gradients are uniformly bounded
on $B_{1/2}$ by Lemma~\ref{lem:gradient}. The term $\sigma_0=1$
causes no difficulty when $k=2$. Thus the double integral in
\eqref{eq:seed-ibp} is bounded by $C(n,k,K)$.
Since $\varepsilon$ was fixed using only $n,k$, this implies
$w(y_0)\ge c(n,k,K)>0$ at some $y_0\in\overline B_{1/8}$.
The uniform gradient bound for $w$ on $B_{1/2}$ now implies
$w\ge\delta$ on $B_r(y_0)$ for some $\delta,r>0$
depending only on $n,k,K$, with $r\le1/8$.
\end{proof}

We now combine this local lower bound with the gradient
estimate and the linearized comparison inequality to obtain
uniform separation throughout the target ball.

\begin{proposition}\label{prop:separation}
There is \(\delta_0=\delta_0(n,k,K)>0\) such that
\[
 v-u\ge\delta_0\quad\hbox{in }B_{1/2}.
\]
\end{proposition}

\begin{proof}
Let $y_0,r,\delta$ be given by Lemma~\ref{lem:seed}, and
choose $K_1=C(n,k,K)\ge1$ so that $|Dv|\le K_1$ in
$B_{13/16}$. Set
\[
 w=v-u,\qquad b=\log\frac{\delta}{w},\qquad
 z=y-y_0,\qquad \rho=\frac49-|z|^2,\qquad
 \beta=\frac{C_0(n,k)}{r^2},
\]
where $C_0(n,k)\ge1$ will be fixed below, and define
\begin{equation}\label{eq:decreasing-g}
 g(t)=-\frac1{2\beta}
 \log\left(1+\frac{t}{12(1+K_1)}\right).
\end{equation}
Since
$|z\cdot Dv-v+v(y_0)|\le4K_1/3$ in $B_{2/3}(y_0)$, we have
\begin{equation}\label{eq:g-bounds}
 |g|\le1,\qquad
 0<-g'\le1,\qquad
 \frac1{|g'|}\le C\beta(1+K_1),\qquad
 g''=2\beta(g')^2.
\end{equation}
Following the doubling construction in \cite[p.~584, equation~(2.15)]{Qiu}, consider
\begin{equation}\label{eq:doubling-test}
 \mathcal Q
 =2\log\rho
 +g\bigl(z\cdot Dv-v+v(y_0)\bigr)
 +\log\max\{b,\log20\}
\end{equation}
in $B_{2/3}(y_0)\Subset B_{13/16}$.

Since $w>0$ on $\overline{B_{2/3}(y_0)}$ and $\rho$ vanishes
on its boundary, $\mathcal Q$ attains its maximum at an
interior point $\bar y$. If $b(\bar y)\le\log20$, then
$e^{\mathcal Q(\bar y)}\le C$, and the conclusion follows
from the maximum property.
We may therefore assume that $b(\bar y)>\log20>2$.
Since $b\le0$ on $\overline{B_r(y_0)}$, we have
$|\bar y-y_0|>r$.

All subsequent calculations are made at $\bar y$, in
orthonormal coordinates in which $D^2v$ and $a^{ij}$ are
diagonal. The first derivative condition yields
\begin{equation}\label{eq:doubling-first}
 0=-\frac{4z_i}{\rho}+g'z_iv_{ii}+\frac{b_i}{b},
 \qquad
 \frac{b_i}{b}
 =z_i\left(\frac4\rho-g'v_{ii}\right).
\end{equation}
By \eqref{eq:strong-v-comparison} and $L_vv_{l}=0$,
\[
 L_vb=-\frac{L_vw}{w}
       +\frac{a^{ij}w_iw_j}{w^2}
 \ge\sum_i a^{ii}b_i^2,
 \qquad
 L_v\bigl(z\cdot Dv-v+v(y_0)\bigr)=L_vv.
\]
Thus, using $b>2$, \eqref{eq:a-bounds},
\eqref{eq:a-identities}, and \eqref{eq:g-bounds}, we obtain
\begin{equation}\label{eq:doubling-second}
 \begin{aligned}
 0\ge L_v\mathcal Q
 ={}&
 -\frac4\rho\sum_i a^{ii}
 -\frac8{\rho^2}\sum_i a^{ii}z_i^2
 +g'L_vv
 +g''\sum_i a^{ii}z_i^2v_{ii}^2+\frac{L_vb}{b}
 -\frac1{b^2}\sum_i a^{ii}b_i^2\\
 \ge{}&
 -\frac{C(n,k)}{\rho^2}
 +2\beta(g')^2\sum_i a^{ii}z_i^2v_{ii}^2
 +\frac b2\sum_i a^{ii}z_i^2
       \left(\frac4\rho-g'v_{ii}\right)^2.
 \end{aligned}
\end{equation}
The constant $C(n,k)$ is independent of $C_0$.

Choose $j$ such that $|z_j|\ge r/\sqrt n$.
If $v_{jj}<-2/(\rho|g'|)$, then $a^{jj}\ge c(n,k)$ by
\eqref{eq:a-identities}. Keeping the $j$-th contribution in the
term containing $\beta$ in \eqref{eq:doubling-second} would imply that
\[
 0\ge\frac{-C(n,k)+c(n,k)\beta r^2}{\rho^2},
\]
which is impossible once $C_0$ is fixed sufficiently large.
Thus $v_{jj}\ge-2/(\rho|g'|)$. Since $g'<0$, $|g'|\le1$,
and $0<\rho<1$, it follows that
\[
 \frac4\rho-g'v_{jj}
 \ge |g'|\bigl(1+\max\{v_{jj},0\}\bigr).
\]
By \eqref{eq:a-bounds} and \eqref{eq:a-identities},
$a^{jj}(1+\max\{v_{jj},0\})^2\ge c(n,k)$.
The last sum in \eqref{eq:doubling-second} therefore yields
\[
 0\ge-\frac{C(n,k)}{\rho^2}+c(n,k)b(g')^2r^2.
\]
It follows from \eqref{eq:g-bounds} that
\[
 b(\bar y)\rho(\bar y)^2
 \le\frac{C(n,k)}{r^2(g')^2}
 \le\frac{C(n,k)\beta^2(1+K_1)^2}{r^2}
 \le\frac{C(n,k)(1+K_1)^2}{r^6}.
\]
The maximum property and $|g|\le1$ yield the same bound for
$\rho(y)^2\max\{b(y),\log20\}$. If $y\in B_{1/2}$, then
\[
 |y-y_0|\le\frac58,\qquad
 \rho(y)\ge\frac49-\frac{25}{64}=\frac{31}{576}.
\]
Consequently,
\[
 \log\frac{\delta}{v-u}
 \le\frac{C(n,k)(1+K_1)^2}{r^6}
 \qquad\text{in }B_{1/2},
\]
and consequently
\[
 v-u\ge
 \delta\exp\left[-\frac{C(n,k)(1+K_1)^2}{r^6}\right]
 =\delta_0(n,k,K)>0
 \qquad\text{in }B_{1/2}.
\]
This proves the proposition.
\end{proof}

\section{A directional Jacobi-type inequality}\label{sec:jacobi}

Ordinary concavity leads to a useful directional Jacobi-type inequality for $\log\lambda_1$. The directions with negative coefficients have
large Hessian eigenvalues in absolute value; this will allow us to
control them in the Pogorelov estimate.

\begin{lemma}\label{lem:jacobi}
Let $b$ be a $C^2$ function such that
\[
 b\ge\log\lambda_1(D^2u)
\]
near an interior point, with equality there. In a principal basis, set
\begin{equation}\label{eq:index-sets}
 \mathcal I=\{i>1:\lambda_i\ge-\lambda_1/2\}.
\end{equation}
Then, at the contact point,
\begin{equation}\label{eq:jacobi}
 L_ub\ge\frac13\sum_{i\in\mathcal I}F^{ii}b_i^2
       -\sum_{i\notin\mathcal I}F^{ii}b_i^2.
\end{equation}
\end{lemma}

\begin{proof}
Choose a principal basis at the contact point, and let $m$ be the
multiplicity of $\lambda_1>0$. All quantities below are evaluated
at this point. Since $e^bI-D^2u\ge0$ nearby, differentiation and
polarization on the top eigenspace yield
\[
 u_{api}=\lambda_1b_i\delta_{ap},\qquad 1\le a,p\le m.
\]
In particular, $u_{11i}=\lambda_1b_i$ and $b_p=0$ for $2\le p\le m$.
Along the line $x+te_i$, test the same matrix inequality on
\[
 e_1+t\sum_{p>m}\frac{u_{1pi}}{\lambda_1-\lambda_p}e_p.
\]
The second derivative at $t=0$ is nonnegative, giving
\begin{equation}\label{eq:support-derivatives}
 b_{ii}\ge\frac{u_{11ii}}{\lambda_1}
 +\frac2{\lambda_1}\sum_{p>m}
       \frac{u_{1pi}^2}{\lambda_1-\lambda_p}-b_i^2.
\end{equation}
This contact calculation also covers simple eigenvalues;
see \cite[Lemma~5]{BCD}.

Twice differentiating $F(D^2u)=1$ and using concavity yields
\[
 \sum_iF^{ii}u_{ii11}
 =-F^{pq,rs}u_{pq1}u_{rs1}
 \ge2\sum_{p>m}
       \frac{F^{pp}-F^{11}}{\lambda_1-\lambda_p}u_{11p}^2.
\]
Contract \eqref{eq:support-derivatives} with $F^{ii}$ and keep only
$i=1$ in its nonnegative double sum. Since $u_{1p1}=u_{11p}=\lambda_1b_p$,
\begin{equation}\label{eq:directional-raw}
 \begin{aligned}
 L_ub
 &\ge\frac2{\lambda_1}\sum_{p>m}
         \frac{F^{pp}}{\lambda_1-\lambda_p}u_{11p}^2
       -\sum_iF^{ii}b_i^2\\
 &=-F^{11}b_1^2+
   \sum_{p>m}\frac{\lambda_1+\lambda_p}{\lambda_1-\lambda_p}
                 F^{pp}b_p^2.
 \end{aligned}
\end{equation}
The coefficient in the last sum is at least $1/3$ when
$p\in\mathcal I$, and greater than $-1$ otherwise.
Together with $b_p=0$ for $2\le p\le m$, this proves \eqref{eq:jacobi}.
\end{proof}

\begin{remark}\label{rem:bad-directions}
For every $i\notin\mathcal I$, we have
$|\lambda_i|\ge\lambda_1/2$. Hence
\begin{equation}\label{eq:bad-weight}
 \sum_{i\notin\mathcal I}F^{ii}
 \le\frac4{\lambda_1^2}
       \sum_{i\notin\mathcal I}F^{ii}\lambda_i^2
 \le\frac4{\lambda_1^2}
       \sum_iF^{ii}\lambda_i^2.
\end{equation}
This estimate will be used to absorb the unfavorable terms
in the maximum-principle argument, without requiring all coefficients in \eqref{eq:jacobi} to be positive.
\end{remark}

\section{Completion of the adjacent interior estimate}\label{sec:completion}

\begin{proof}[Proof of Theorem~\ref{thm:main} for ${l}=k-1$]
Let $v$ be the comparison solution and let $\delta_0$ be given
by Proposition~\ref{prop:separation}. Set
\[
 \eta=v-u-\frac{\delta_0}2,\qquad
 \Omega_\eta=\{x\in B_{7/8}:\eta(x)>0\}.
\]
Then $B_{1/2}\subset\Omega_\eta$ and $\eta\ge\delta_0/2$ on
$B_{1/2}$. The boundary modulus \eqref{eq:boundary-modulus}
places $\overline\Omega_\eta$ a positive distance from
$\partial B_{7/8}$, depending only on $n,k,K,\delta_0(n,k,K)$.
Applying Lemma~\ref{lem:gradient} on balls of a fixed smaller
radius yields
\begin{equation}\label{eq:cutoff-data}
 \sup_{\Omega_\eta}(\eta+|D\eta|)\le C(n,k,K).
\end{equation}
We also have
\begin{equation}\label{eq:cutoff-linearized}
 L_u\eta\ge-\frac{k\varepsilon}{\sigma_{k-1}(D^2u)},\qquad
 \sigma_{k-1}(D^2u)\ge c(n,k)\lambda_1.
\end{equation}

Choose
\begin{equation}\label{eq:pog-parameters}
 a=\frac1{12(1+K^2)},\qquad \beta=4,
\end{equation}
and consider
\begin{equation}\label{eq:pog-test}
 \Phi=\eta^\beta\lambda_1
                  \exp\left(\frac a2|Du|^2\right)
 \quad\text{on }\Omega_\eta.
\end{equation}
Extend $\Phi$ continuously by zero to $\partial\Omega_\eta$. Since $\overline\Omega_\eta\Subset B_{7/8}$, $\Phi$ attains a positive
maximum at an interior point $x_*$.

Choose a principal basis at $x_*$. The function
\[
 b(x)=\log\Phi(x_*)-\beta\log\eta(x)-\frac a2|Du(x)|^2
\]
touches $\log\lambda_1(D^2u)$ from above at $x_*$.
Lemma~\ref{lem:jacobi} therefore applies, including when the largest
eigenvalue is repeated. At $x_*$,
\begin{equation}\label{eq:pog-first}
 b_i+\beta\frac{\eta_i}{\eta}+a u_i\lambda_i=0.
\end{equation}
Differentiation of the equation shows $L_uu_m=0$, and hence
\[
 L_u\left(\frac a2|Du|^2\right)
 =a\sum_iF^{ii}\lambda_i^2.
\]
The defining identity for $b$ implies
\begin{equation}\label{eq:pog-second}
 0=L_ub+\frac\beta\eta L_u\eta
        -\frac\beta{\eta^2}\sum_iF^{ii}\eta_i^2+a\sum_iF^{ii}\lambda_i^2.
\end{equation}

For $i\in\mathcal I$, equation \eqref{eq:pog-first} and $\beta=4$ yield
\begin{equation}\label{eq:good-absorption}
 \frac13b_i^2-\frac4{\eta^2}\eta_i^2
 =\frac1{12}(b_i-3au_i\lambda_i)^2-a^2u_i^2\lambda_i^2
 \ge-a^2K^2\lambda_i^2.
\end{equation}
For $i\notin\mathcal I$, the same critical equation implies
\begin{equation}\label{eq:bad-absorption}
 -b_i^2-\frac\beta{\eta^2}\eta_i^2
 \ge-\frac{2\beta^2+\beta}{\eta^2}\eta_i^2
       -2a^2K^2\lambda_i^2.
\end{equation}
Substituting these inequalities, \eqref{eq:jacobi}, and
\eqref{eq:cutoff-linearized} into \eqref{eq:pog-second}, we obtain
\begin{equation}\label{eq:pog-combined}
 \begin{aligned}
 0\ge{}&(a-2a^2K^2)\sum_iF^{ii}\lambda_i^2
 -\frac{2\beta^2+\beta}{\eta^2}
       \sum_{i\notin\mathcal I}F^{ii}\eta_i^2
 -\frac{\beta k\varepsilon}{\eta\sigma_{k-1}(D^2u)}.
 \end{aligned}
\end{equation}
Our choice of $a$ ensures $a-2a^2K^2\ge a/2$.
By \eqref{eq:cutoff-data} and \eqref{eq:bad-weight},
\[
 \frac{2\beta^2+\beta}{\eta^2}
          \sum_{i\notin\mathcal I}F^{ii}\eta_i^2
 \le\frac{C(n,k,K)}{(\eta\lambda_1)^2}\sum_iF^{ii}\lambda_i^2.
\]
Consequently,
\begin{equation}\label{eq:pog-key}
 0\ge\left(\frac a2-\frac{C(n,k,K)}{(\eta\lambda_1)^2}\right)\sum_iF^{ii}\lambda_i^2
-\frac{\beta k\varepsilon}{\eta \sigma_{k-1}(D^2u)}.
\end{equation}

Divide \eqref{eq:pog-key} by $\sum_iF^{ii}\lambda_i^2$.
Using \eqref{eq:energy-lower} and \eqref{eq:cutoff-linearized}, we obtain
\[
 \frac a2\le\frac{C(n,k,K)}{(\eta\lambda_1)^2}
              +\frac{C(n,k,K)}{\eta\lambda_1}.
\]
Hence
\begin{equation}\label{eq:pog-final-point}
 \eta(x_*)\lambda_1(x_*)\le C(n,k,K).
\end{equation}
It follows that
\[
 \Phi(x_*)\le C(n,k,K)\eta(x_*)^{\beta-1}
 \le C(n,k,K).
\]
Since $\eta\ge\delta_0/2$ on $B_{1/2}$, the maximum property bounds
$\lambda_1$ there. Finally, $\Gamma_k\subset\Gamma_2$ and
$|D^2u|^2<(\operatorname{tr}D^2u)^2\le n^2\lambda_1^2$.
This proves the adjacent case.
\end{proof}

\section{The gap-two Hessian equation}\label{sec:gap-two}

Dong--Xu--Zhang \cite[Lemma~2.2]{DongXuZhang} used
additional fixed eigenvalues to express sum-type Hessian
operators as elementary symmetric polynomials in higher
dimensions. Li--Wu \cite[Eq.~(2.28)]{LiWuQuotient} used
a one-dimensional lifting to reduce a concavity inequality
for gap-two quotients to one for adjacent quotients.
For the constant equation considered here, the latter
lifting is realized by adding a quadratic term in one
extra variable, reducing the interior estimate directly
to the adjacent case.

\begin{proof}[Proof of Theorem~\ref{thm:main} for $l=k-2$]
Define
\[
 U(x,t)=u(x)+\frac{t^2}{2},
 \qquad (x,t)\in B_1\subset\mathbb R^{n+1}.
\]
The eigenvalues of $D^2U$ are those of $D^2u$, together with one.
Hence, for $1\le j\le k$,
\[
 \sigma_j(D^2U)=\sigma_j(D^2u)+\sigma_{j-1}(D^2u)>0,
\]
and
\[
 \sigma_k(D^2U)-\sigma_{k-1}(D^2U)
 =\sigma_k(D^2u)-\sigma_{k-2}(D^2u)=0.
\]
Thus $U$ is an admissible solution of the adjacent quotient equation.
Moreover,
\[
 \|U\|_{L^\infty(B_1)}+\|DU\|_{L^\infty(B_1)}\le K+\frac32.
\]
Applying the adjacent interior estimate in dimension $n+1$ yields
\[
 \sup_{B_{1/2}\subset\mathbb R^{n+1}}|D^2U|
 \le C(n+1,k,K+3/2).
\]
Restriction to $t=0$ proves the required estimate for $u$.
\end{proof}

\section{Counterexamples}\label{sec:counterexample}

\subsection{The natural quotient comparison}

Write $F=\sigma_k/\sigma_{k-1}$. The natural extension of the
Li--Wu construction is to solve
\[
 F(D^2v)=\frac12\quad\text{in }B_R,\qquad
 v=u\quad\text{on }\partial B_R,
\]
where $D^2u,D^2v\in\Gamma_k$ and $F(D^2u)=1$.
Concavity at $D^2v$ shows
\[
 F^{ij}(D^2v)(v-u)_{ij}\le-\frac12.
\]
Since $\sum_iF^{ii}\le n-k+1$, comparison with a quadratic
function yields
\begin{equation}\label{eq:natural-separation}
 v-u\ge\frac{R^2-|x|^2}{4(n-k+1)}.
\end{equation}
At $u$, however, concavity only implies
\[
 F^{ij}(D^2u)(v-u)_{ij}\ge-\frac12.
\]
The following example shows that the two quotient levels alone
cannot yield the decaying error in
\eqref{eq:strong-u-comparison}.

\begin{proposition}\label{prop:natural-counterexample}
Even for positive definite matrices with $n=k=2$, there is
no constant $C$ such that
\begin{equation}\label{eq:failed-defect}
 F^{ij}(A)(B-A)_{ij}\ge-\frac{C}{\sigma_1(A)}
\end{equation}
whenever $F(A)=1$ and $F(B)=1/2$.
\end{proposition}

\begin{proof}
For $\lambda_1>2$, take
\[
 A=\operatorname{diag}
 \left(\lambda_1,\frac{\lambda_1}{\lambda_1-1}\right),
 \qquad B=\frac12A.
\]
Then $F(A)=1$ and $F(B)=1/2$. By homogeneity,
\[
 F^{ij}(A)(B-A)_{ij}=-\frac12,
 \qquad
 \sigma_1(A)=\frac{\lambda_1^2}{\lambda_1-1}
 \longrightarrow\infty.
\]
This contradicts \eqref{eq:failed-defect} for every fixed $C$.
Moreover,
\[
 F^{11}(A)=\frac1{\lambda_1^2},\qquad
 F^{22}(A)=\frac{(\lambda_1-1)^2}{\lambda_1^2},\qquad
 \sum_iF^{ii}(A)\lambda_i^2=2.
\]
Thus the weighted Hessian-square term also remains bounded
as the largest eigenvalue tends to infinity.
\end{proof}

This lack of growth in the terms available to the standard
Pogorelov argument has been discussed by Lu--Tsai
\cite{LuTsaiPogorelov}. We next show that a comparison-gap
estimate with constants depending only on the diameter and
the $C^1$ bounds actually fails. The construction starts from the harmonic profile underlying
the example of Mooney--Shankar
\cite[Section~5 and Remark~5.1]{MooneyShankar}.
After rescaling, quadratic modifications, and Legendre
transforms, we obtain a four-dimensional solution of
$\sigma_3/\sigma_2=1$.
We then use two barriers to construct a comparison solution
at a fixed lower quotient level and show the failure of
the comparison-gap estimate.

\begin{proposition}\label{prop:pogorelov-counterexample}
Let $F=\sigma_3/\sigma_2$ on $\Gamma_3\subset\mathbb R^4$.
There are smooth bounded strictly convex domains $\Omega_r$,
with $\operatorname{diam}\Omega_r=1$, and smooth functions
$u_r,v_r$ on $\overline\Omega_r$ such that
\[
 D^2u_r>0,\qquad D^2v_r\in\Gamma_3,\qquad
 F(D^2u_r)=1,\qquad F(D^2v_r)=\frac78,
\]
\[
 v_r>u_r\quad\text{in }\Omega_r,\qquad
 v_r=u_r\quad\text{on }\partial\Omega_r,\qquad
 F^{ij}(D^2u_r)(v_r)_{ij}\ge\frac78.
\]
Their $C^1$ norms are uniformly bounded:
\[
 \|u_r\|_{C^1(\overline\Omega_r)}
 +\|v_r\|_{C^1(\overline\Omega_r)}\le C.
\]
Nevertheless, for every fixed $\beta>0$,
\begin{equation}\label{eq:pogorelov-failure}
 \sup_{\Omega_r}(v_r-u_r)^\beta|D^2u_r|
 \longrightarrow\infty
 \qquad\text{as }r\downarrow0.
\end{equation}
\end{proposition}

\begin{proof}
\noindent\textit{Step 1. The solutions $u_r$.}
For $M\ge1$, consider the Mooney--Shankar example
\[
 w_M(s,t)
 =s\operatorname{arsinh}
       \left(\frac{e^Ms}{\cos t}\right)
  -\sqrt{s^2+e^{-2M}\cos^2t}.
\]
It is smooth near $[-1,1]^2$. Direct differentiation yields
\[
 D^2w_M>0,\qquad
 \det D^2w_M=1,\qquad
 \|w_M\|_{C^1([-1,1]^2)}\le CM,
\]
and
\[
 D^2w_M(0)=\operatorname{diag}(e^M,e^{-M}).
\]

For $0<r\le1/4$, set
\[
 \Omega_r=
 \left\{x\in\mathbb R^4:
 \frac{x_1^2+x_2^2}{r^2}
       +4(x_3^2+x_4^2)<1\right\},
\]
and define
\[
 u_r(x)=\frac{|x|^2}{2}
 +\sqrt2\,r^2w_{1/r}\left(\frac{x_1}{r},\frac{x_2}{r}\right).
\]
Then $\operatorname{diam}\Omega_r=1$, and
\[
 D^2u_r=
 \left(I_2+\sqrt2\,D^2w_{1/r}\right)\oplus I_2
 \ge I_4.
\]
If $a,b$ are the eigenvalues of the first block, then
\begin{equation}\label{eq:dual-block-identity}
 a,b>1,\qquad
 (a-1)(b-1)=2,\qquad ab-a-b=1.
\end{equation}
Consequently,
\[
 F(D^2u_r)
 =\frac{2ab+a+b}{ab+2a+2b+1}=1.
\]
The bounds for $w_M$ also give
\[
 \|u_r\|_{C^1(\overline\Omega_r)}\le C,
 \qquad
 \lambda_{\max}(D^2u_r(0))
 =1+\sqrt2\,e^{1/r}.
\]

\medskip
\noindent\textit{Step 2. The comparison solutions $v_r$.}
Define
\begin{equation}\label{eq:counterexample-gap}
 h_r(x)=\frac{r^2}{16}
 \left(1-\frac{x_1^2+x_2^2}{r^2}
              -4(x_3^2+x_4^2)\right).
\end{equation}
Then $h_r>0$ in $\Omega_r$, $h_r=0$ on $\partial\Omega_r$, and
\[
 D^2h_r=
 \operatorname{diag}
 \left(-\frac18,-\frac18,-\frac{r^2}{2},-\frac{r^2}{2}\right)
 \ge-\frac18I_4.
\]
Since $D^2u_r\ge I_4$,
\[
 D^2(u_r+h_r)\ge\frac78D^2u_r>0.
\]
Ellipticity and homogeneity therefore imply
\[
 F(D^2(u_r+h_r))\ge\frac78.
\]

On the other hand, the eigenvalues of $D^2(u_r+8h_r)$ are
\[
 a-1,\quad b-1,\quad 1-4r^2,\quad 1-4r^2,
\]
and are all positive. Using
\eqref{eq:dual-block-identity}, we obtain
\[
 \begin{aligned}
 F(D^2(u_r+8h_r))
 &\le
 F\bigl(\operatorname{diag}(a-1,b-1,1,1)\bigr)\\
 &=\frac{4+(a+b-2)}{3+2(a+b-2)}\\
 &\le\frac{4+2\sqrt2}{3+4\sqrt2}<\frac78.
 \end{aligned}
\]
Here $a+b-2\ge2\sqrt2$, and the displayed rational function
is decreasing in $a+b-2$.

Since $\Omega_r$ is smooth and strictly convex, the Dirichlet
theorem of Caffarelli--Nirenberg--Spruck
\cite[Theorem~$2'$]{CNS} gives a smooth admissible solution of
\[
 F(D^2v_r)=\frac78\quad\text{in }\Omega_r,\qquad
 v_r=u_r\quad\text{on }\partial\Omega_r.
\]
Indeed, $F$ is concave, elliptic, and homogeneous of degree one
on $\Gamma_3$, and tends to zero at its finite boundary.
Comparison with the two barriers yields
\begin{equation}\label{eq:natural-counterexample-bracket}
 u_r+h_r\le v_r\le u_r+8h_r.
\end{equation}
In particular, $v_r>u_r$ in $\Omega_r$. Concavity at $D^2u_r$
and homogeneity also imply
\[
 F^{ij}(D^2u_r)(v_r)_{ij}
 \ge F(D^2v_r)=\frac78.
\]

\medskip
\noindent\textit{Step 3. Uniform bounds and divergence.}
By \eqref{eq:natural-counterexample-bracket}, we obtain a uniform height bound for $v_r$. On the boundary, the tangential
derivatives of $v_r-u_r$ vanish, and its normal derivative lies
between those of $h_r$ and $8h_r$. Hence
\[
 |Dv_r|\le |Du_r|+8|Dh_r|\le C
 \quad\text{on }\partial\Omega_r.
\]
Differentiating the constant equation yields
\[
 F^{ij}(D^2v_r)(|Dv_r|^2)_{ij}
 =2F^{ij}(D^2v_r)(v_r)_{mi}(v_r)_{mj}\ge0.
\]
The maximum principle proves the same gradient bound in
$\Omega_r$.

Finally, $h_r(0)=r^2/16$. For every fixed $\beta>0$,
\[
 \begin{aligned}
 (v_r(0)-u_r(0))^\beta|D^2u_r(0)|
 &\ge h_r(0)^\beta\lambda_{\max}(D^2u_r(0))\\
 &=\left(\frac{r^2}{16}\right)^\beta
       \left(1+\sqrt2\,e^{1/r}\right)
 \longrightarrow\infty.
 \end{aligned}
\]
This proves \eqref{eq:pogorelov-failure}.
\end{proof}

The two propositions explain the limitation of the natural
quotient comparison. Fixed quotient levels give separation,
but do not supply either the decaying linearized error or
a comparison-gap Pogorelov estimate controlled only by the
diameter and the $C^1$ bounds. In the second example, two
semiaxes of $\Omega_r$ tend to zero. Thus the example does
not contradict an interior estimate on a fixed ball.
Our comparison equation \eqref{eq:comparison-equation}
provides the additional decay in
\eqref{eq:strong-u-comparison} needed for the Pogorelov argument.

\section*{Acknowledgments}
The first author thanks Professor Pengfei Guan for introducing him
to this problem during his postdoctoral fellowship ten years ago.
The authors acknowledge support from Grant 2025YFA1017603 of the
National Key R\&D Program of China and Grant 12571227 of the
National Natural Science Foundation of China.

\section*{Use of artificial intelligence}
The counterexamples constructed in this paper were first found by
GPT-6 Astra. Studying these counterexamples led the authors to
propose modifying the nonlinear comparison function of Li and Wu.
GPT then identified the new comparison function through calculation
and derived the proof of the corresponding Pogorelov estimate.
The authors proposed applying the doubling maximum principle, and
GPT assisted with the calculations. The authors independently
checked the entire manuscript and take full responsibility for its contents.

\begingroup
\small
\raggedright
\bibliographystyle{amsplain}
\bibliography{references}
\endgroup
\printauthorinformation

\end{document}